\documentclass[a4paper,12pt]{article}
\usepackage{amsmath,amsthm,amssymb}
\usepackage[dvipdfmx]{graphicx}

\numberwithin{equation}{section}

\newtheorem{thm}{Theorem}[section]

\newtheorem{lmm}[thm]{Lemma}

\newtheorem{prp}[thm]{Proposition}

\theoremstyle{definition}
\newtheorem{dfn}[thm]{Definition}

\theoremstyle{remark}
\newtheorem{rem}{Remark}

\newcommand{\R}{\mathbb{R}}
\newcommand{\p}{\partial}
\newcommand{\eps}{\varepsilon}

\title{Strong instability of standing waves \\
for nonlinear Schr\"{o}dinger equations \\
with a delta potential}

\author{Masahito \textsc{Ohta}\footnote{
Department of Mathematics, Tokyo University of Science, 
1-3 Kagurazaka, Shinjuku-ku, Tokyo 162-8601, Japan. 
\newline e-mail: \texttt{mohta@rs.tus.ac.jp}}
~and Takahiro \textsc{Yamaguchi}\footnote{
Department of Mathematics, Tokyo University of Science, 
1-3 Kagurazaka, Shinjuku-ku, Tokyo 162-8601, Japan.
\endgraf e-mail: \texttt{1113621@ed.tus.ac.jp}}}

\date{}

\begin{document}

\maketitle

\begin{abstract}    
We study strong instability (instability by blowup) 
of standing wave solutions for a nonlinear Schr\"{o}dinger equation 
with an attractive delta potential 
and $L^2$-supercritical power nonlinearity in one space dimension. 
We also compare our sufficient condition on strong instability 
with some known results on orbital instability.  
\end{abstract}

\section{Introduction}

In our previous paper \cite{OY}, 
we studied the strong instability (instability by blowup) 
of standing wave solutions 
$e^{i\omega t}\phi_{\omega}(x)$ 
for the following nonlinear Schr\"{o}dinger equation 
with double power nonlinearity: 
\begin{align}\label{nls0}
i\p_t u=-\Delta u-a|u|^{p-1}u-b|u|^{q-1}u,
\quad (t,x)\in \R\times \R^N, 
\end{align}
where $a$ and $b$ are positive constants, 
$1<p<1+4/N<q<2^*-1$. 
Here, $2^*$ is defined by $2^*=2N/(N-2)$ if $N\ge 3$, 
and $2^*=\infty$ if $N=1,2$. 
For ground states $\phi_{\omega}$, 
we proved that the standing wave solution 
$e^{i\omega t}\phi_{\omega}(x)$ of \eqref{nls0} is strongly unstable 
for sufficiently large $\omega$. 
Moreover, we announced in \cite{OY} that our method of proof is not 
restricted to the double power case \eqref{nls0}, 
but is also applicable to other type of nonlinear Schr\"odinger equations. 

In this paper,  
we consider the following nonlinear Schr\"{o}dinger equation 
with a delta potential in one space dimension: 
\begin{align}\label{nls}
i\p_t u=-\p_x^2 u-\gamma \delta (x)u-|u|^{p-1}u,
\quad(t,x) \in \R \times \R,
\end{align}
where $\gamma\in \R$ is a constant, 
$\delta(x)$ is the delta measure at the origin, and $1<p<\infty$. 
The equations of the form \eqref{nls} arise in a wide variety of physical models 
with a point defect on the line, 
and have been studied by many authors 
(see, e.g., \cite{FJ, FOO, GHW, HMZ1, HMZ2, HZ, KO, LFF} and references therein). 

We study the strong instability of standing wave solutions 
$e^{i \omega t}\phi_{\omega}(x)$ of \eqref{nls}, 
where $\omega>\gamma^2/4$, and 
\begin{align}\label{for}
\phi_{\omega}(x)=
\bigg\{\frac{(p+1)\omega}{2}{\rm sech}^2
\left(\frac{(p-1)\sqrt{\omega}}{2}|x|+\tanh^{-1}
\left(\frac{\gamma}{2\sqrt{\omega}}\right) \right) 
\bigg\}^{\frac{1}{p-1}}, 
\end{align}
which is a unique positive solution of
\begin{align}\label{sp}
-\p_x^2\phi +\omega \phi -\gamma \delta(x)\phi-|\phi|^{p-1}\phi=0, \quad x\in \R. 
\end{align}

The well-posedness of the Cauchy problem for \eqref{nls} in the energy space $H^1(\R)$ 
follows from an abstract result in Cazenave \cite{caz}
(see Theorem 3.7.1 and Corollary 3.3.11 in \cite{caz}, 
and also Section 2 of \cite{FOO}). 

\begin{prp}
For any $u_0\in H^1(\R)$ there exist $T_{\max}=T_{\max}(u_0) \in  (0,\infty]$ 
and a unique solution
$u\in C([0,T_{\max}),H^1(\R))$ with $u(0)=u_0$ such that either $T_{\max}=\infty$ (global existence)
or $T_{\max}<\infty$ and $\displaystyle\lim_{t\to T_{\max}}\| \p_x u(t)\|_{L^2}=\infty$ (finite time blowup).
Furthermore, the solution $u(t)$ satisfies
\begin{align}\label{conservation}
E(u(t))=E(u_0), \quad \|u(t)\|_{L^2}^2=\|u_0\|_{L^2}^2
\end{align}
for all $t\in [0,T_{\max})$, 
where the energy $E$ is defined by
\begin{align*}
E(v)=\frac{1}{2}\|\p_x v\|_{L^2}^2-\frac{\gamma}{2}|v(0)|^2
-\frac{1}{p+1}\| v\|_{L^{p+1}}^{p+1}.
\end{align*}
\end{prp}

Here, we give the definitions of stability and instability of standing waves. 

\begin{dfn}
We say that the standing wave solution $e^{i\omega t}\phi_\omega$ of \eqref{nls} 
is {\em orbitally stable} if for any $\eps>0$ there exists $\delta>0$ such that 
if $\|u_0-\phi_\omega \|_{H^1}<\delta$, 
then the solution $u(t)$ of \eqref{nls} with $u(0)=u_0$ exists globally and satisfies 
$$\sup_{t \geq 0}\inf_{\theta \in \R}\|u(t)-e^{i\theta}\phi_{\omega}\|_{H^1}<\eps.$$
Otherwise, $e^{i\omega t}\phi_{\omega}$ is said to be {\em orbitally unstable}.
\end{dfn}

\begin{dfn}
We say that $e^{i\omega t}\phi_\omega$ is {\em strongly unstable} 
if for any $\eps>0$ there exists $u_0 \in H^1(\R)$ 
such that $\|u_0-\phi_{\omega}\|_{H^1}<\eps$ and 
the solution $u(t)$ of \eqref{nls} with $u(0)=u_0$ blows up in finite time.
\end{dfn}

Before we state our main result, 
we recall some known results. 
First, we consider the nonlinear Schr\"odinger equation without potential: 
\begin{align}\label{nls1}
i\p_t u=-\p_x^2 u-|u|^{p-1}u,  
\quad (t,x)\in \R \times \R. 
\end{align}
Let $1<p<\infty$, $\omega>0$ and 
\begin{align*}
\varphi_{\omega}(x)=\bigg\{ \frac{(p+1)\omega}{2}{\rm sech}^2
\left(\frac{(p-1)\sqrt{\omega}}{2} x\right)  \bigg\}^{\frac{1}{p-1}}. 
\end{align*}
When $1<p<5$, the standing wave solution $e^{i\omega t}\varphi_{\omega}$ of \eqref{nls1}
is orbitally stable for all $\omega>0$ (see \cite{CL}). 
When $p\ge 5$, 
$e^{i\omega t}\varphi_{\omega}$ is strongly unstable for all $\omega>0$ 
(see \cite{BC} and also \cite{caz}). 

Next, we consider the attractive potential case $\gamma>0$ in \eqref{nls}, 
which was first studied by Goodman, Holmes and Weinstein \cite{GHW} 
for the case $p=3$, 
and then by Fukuizumi, Ohta and Ozawa \cite{FOO} for $1<p<\infty$. 
The following is proved in \cite{FOO}. 

\begin{prp}[\cite{FOO}]\label{prp2}
Let $\gamma>0$ and $\omega>\gamma^2/4$.  
\begin{itemize}
\item[{\rm (i)}] 
When $1<p\le 5$, 
the standing wave solution $e^{i \omega t}\phi_{\omega}$ of \eqref{nls} 
is orbitally stable for any $\omega \in (\gamma^2/4,\infty)$. 
\item[{\rm (ii)}] 
When $p>5$, 
there exists $\omega_0=\omega_0(p,\gamma) \in (\gamma^2/4,\infty)$  such that 
the standing wave solution $e^{i \omega t}\phi_{\omega}$ of \eqref{nls} 
is orbitally stable for any $\omega \in (\gamma^2/4,\omega_0)$, 
and it is orbitally unstable for any $\omega \in (\omega_0,\infty)$. 
Here, $\omega_0(p,\gamma)=\gamma^2/[4\xi_0(p)^2]$ 
and $\xi_0(p)\in (0,1)$ is a unique solution of 
\begin{equation}\label{xi0}
\frac{p-5}{p-1} \int_{\xi}^{1}(1-s^2)^{\frac{2}{p-1}-1}\,ds
=\xi\,(1-\xi^2)^{\frac{2}{p-1}-1} \quad (0<\xi<1). 
\end{equation}
\end{itemize}
\end{prp}

\begin{rem}\label{rem1}
To prove Proposition \ref{prp2}, 
the following sufficient conditions for orbital stability and instability are used 
(see \cite{GSS1, GSS2, sha, SS, wei}). 

Let $p>1$, $\gamma>0$ and $\omega>\gamma^2/4$.  
\begin{enumerate}
\item[(i)]
If $\partial_{\omega}\|\phi_{\omega}\|_{L^2}^{2}>0$ at $\omega=\hat \omega$, 
then $e^{i \hat \omega t}\phi_{\hat \omega}$ is orbitally stable. 
\item[(ii)] 
If $\partial_{\omega}\|\phi_{\omega}\|_{L^2}^{2}<0$ at $\omega=\hat \omega$, 
then $e^{i\hat \omega t}\phi_{\hat \omega}$  is orbitally unstable. 
\end{enumerate}

By the formula \eqref{for}, we have 
\begin{align*}
\|\phi_{\omega}\|_{L^2}^{2}
&=2\int_{0}^{\infty} 
\bigg\{\frac{(p+1)\omega}{2} {\rm sech}^2
\left(\frac{(p-1)\sqrt{\omega}}{2}\,x+\tanh^{-1} \xi  (\omega,\gamma) \right) 
\bigg\}^{\frac{2}{p-1}}dx \\
&=\frac{4}{(p-1)\sqrt{\omega}} \left(\frac{(p+1)\omega}{2}\right)^{\frac{2}{p-1}}
\int_{\tanh^{-1} \xi  (\omega,\gamma)}^{\infty} ({\rm sech}^2 y)^{\frac{2}{p-1}}\,dy,
\end{align*}
where we put 
\begin{equation}\label{xi-omega}
\xi (\omega,\gamma)=\frac{\gamma}{2\sqrt{\omega}}.
\end{equation}
Moreover, for $0<a<1$ and $\beta>0$, we have 
\begin{equation}\label{CI}
\int_{\tanh^{-1}a}^{\infty} ({\rm sech}^2 y)^{\beta}\,dy
=\int_{a}^{1} (1-s^2)^{\beta-1}\,ds.
\end{equation} 
Thus, we obtain 
\begin{align*}
&\|\phi_{\omega}\|_{L^2}^{2}
=\frac{4}{p-1} \left(\frac{p+1}{2}\right)^{\frac{2}{p-1}}
\left(\frac{2}{\gamma}\right)^{\frac{p-5}{p-1}} F\left(\xi(\omega,\gamma)\right), \\
&F(\xi)=\xi^{\frac{p-5}{p-1}} \int_{\xi}^{1}(1-s^2)^{\frac{2}{p-1}-1}\,ds. 
\end{align*}
Then, since $\partial_{\omega}\xi(\omega,\gamma)<0$, 
for $\xi=\xi(\omega,\gamma)$, we see that 
\begin{align*}
&\partial_{\omega}\|\phi_{\omega}\|_{L^2}^{2}<0 \iff F'(\xi)>0 \\
&\iff \frac{p-5}{p-1}\int_{\xi}^{1}(1-s^2)^{\frac{2}{p-1}-1}\,ds
>\xi\, (1-\xi^2)^{\frac{2}{p-1}-1} \\
&\iff \xi<\xi_0(p) 
\iff \omega>\omega_0(p,\gamma). 
\end{align*}
\end{rem} 

\begin{rem}
For the borderline case $\omega=\omega_0$ in Proposition \ref{prp2} (ii), 
the standing wave solution $e^{i \omega_0 t}\phi_{\omega_0}$ of \eqref{nls} 
is orbitally unstable (see \cite{oht4}). 
\end{rem}

Now we state our main result in this paper.

\begin{thm}\label{thm1}
Let $\gamma>0$, $p>5$, $\omega>\gamma^2/4$, 
and let $\phi_{\omega}$ be the function defined by \eqref{for}.  
Let $\xi_1(p)\in (0,1)$ be a unique solution of 
\begin{equation}\label{xi1}
\frac{p-5}{p-1} \int_{\xi}^{1}(1-s^2)^{\frac{2}{p-1}}\,ds
=\xi\,(1-\xi^2)^{\frac{2}{p-1}} \quad (0<\xi<1),
\end{equation}
and define $\omega_1=\omega_1(p,\gamma)=\gamma^2/[4\xi_1(p)^2]$. 
Then, the standing wave solution $e^{i\omega t}\phi_{\omega}$ of \eqref{nls} 
is strongly unstable for all $\omega\in (\omega_1,\infty)$. 
\end{thm}

\begin{rem}
The condition $\omega>\omega_1$ in Theorem \ref{thm1} 
is equivalent to $E(\phi_{\omega})>0$ 
(see Theorem \ref{thm2} below). 
\end{rem}

\begin{rem}
For the repulsive potential case $\gamma<0$, 
it is proved in \cite{LFF} that if $p\ge 5$, 
the standing wave solution $e^{i \omega t}\phi_{\omega}$ of \eqref{nls} 
is strongly unstable for all $\omega \in (\gamma^2/4,\infty)$. 
The situation for the attractive potential case $\gamma>0$ 
is quite different from the case $\gamma<0$, 
and we need a new approach to prove Theorem \ref{thm1}. 
\end{rem}

For $\gamma>0$, $p>1$ and $\omega>\gamma^2/4$, 
we define functionals $S_{\omega}$ and $K_{\omega}$ on $H^1(\R)$ by
\begin{align*}
 &S_{\omega}(v)=\frac{1}{2}\|\p_x v\|_{L^2}^2
 +\frac{\omega}{2}\|v\|_{L^2}^2
 -\frac{\gamma}{2}|v(0)|^2
 -\frac{1}{p+1}\|v\|_{L^{p+1}}^{p+1}, \\
 &K_{\omega}(v)=\|\p_x v\|_{L^2}^2
 +\omega \|v\|_{L^2}^2
 -\gamma |v(0)|^2
 -\|v\|_{L^{p+1}}^{p+1}.
\end{align*}
Note that  \eqref{sp} is equivalent to $S_{\omega}'(\phi)=0$, and 
\begin{align*}
 K_{\omega}(v)=\p_{\lambda} S_{\omega}(\lambda v)\big|_{\lambda=1}
 =\langle S_{\omega}'(v), v \rangle
\end{align*}
is the so-called Nehari functional. 

We denote the set of nontrivial solutions of \eqref{sp} by 
$$\mathcal{A}_{\omega}
=\{v\in H^1(\R^N): S_{\omega}'(v)=0, ~ v\ne 0\},$$ 
and define the set of ground states of \eqref{sp} by 
\begin{align}
\mathcal{G}_{\omega}
=\{\phi\in \mathcal{A}_{\omega}:
S_{\omega}(\phi)\le S_{\omega}(v)
\hspace{1mm} \mbox{for all} \hspace{1mm} 
v\in \mathcal{A}_{\omega}\}. 
\label{gs}
\end{align}

Moreover, consider the minimization problem:
\begin{align}
d(\omega)
=\inf\{S_{\omega}(v):v\in H^1(\R^N), ~ K_{\omega}(v)=0, ~ v\ne 0\}.
\label{gs1}
\end{align}
Then, for any $\omega>\gamma^2/4$, we have 
\begin{align*}
\mathcal{A}_{\omega}
&=\mathcal{G}_{\omega}
=\{\phi\in H^1(\R^N): 
S_{\omega}(\phi)=d(\omega), ~ K_{\omega}(\phi)=0\} \\
&=\{e^{i \theta}\phi_{\omega} : \theta \in \R \},
\end{align*}
where $\phi_{\omega}$ is the function defined by \eqref{for} 
(see \cite{FOO, LFF}). 

On the other hand, 
the proof of finite time blowup for \eqref{nls} relies on the virial identity. 
If $u_0\in \Sigma:=\{v\in H^1(\R): |x|v\in L^2(\R)\}$, 
then the solution $u(t)$ of \eqref{nls} with $u(0)=u_0$ 
belongs to $C([0,T_{\max}),\Sigma)$, and satisfies 
\begin{align} \label{virial}
\frac{d^2}{dt^2}\|xu(t)\|_{L^2}^2=8P(u(t))
\end{align}
for all $t\in [0,T_{\max})$, where 
\begin{align*}
P(v)=\|\p_x v\|_{L^2}^2
-\frac{\gamma}{2}|v(0)|^2
-\frac{\alpha}{p+1}\|v\|_{L^{p+1}}^{p+1}, 
\quad \alpha:=\frac{p-1}{2}. 
\end{align*}
For the proof of the virial identity \eqref{virial}, see Proposition 6 in \cite{LFF}.

Note that for the scaling $v^{\lambda}(x)=\lambda^{1/2}v(\lambda x)$ for $\lambda>0$, 
we have 
\begin{align*}
&\|\p_x v^{\lambda}\|_{L^2}^2=\lambda^2\|\p_x v\|_{L^2}^2, \quad
|v^{\lambda}(0)|^2=\lambda|v(0)|^2, \quad 
\|v^{\lambda}\|_{L^{p+1}}^{p+1}=\lambda^{\alpha}\|v\|_{L^{p+1}}^{p+1}, \\
&\|v^{\lambda}\|_{L^2}^2=\|v\|_{L^2}^2, \quad
P(v)=\p_{\lambda} E(v^{\lambda})\big|_{\lambda=1}.
\end{align*}

The method of Berestycki and Cazenave \cite{BC} 
for the nonlinear Schr\"{o}dinger equations without potential \eqref{nls1} 
is based on the fact that 
$d(\omega)=S_{\omega}(\phi_{\omega})$ can be characterized as 
\begin{align}
d(\omega)
=\inf\{S_{\omega}(v):v\in H^1(\R), ~ P(v)=0, ~ v\ne 0\} 
\label{bc1}
\end{align}
for the case $p\ge 5$. 
Using this fact, it is proved in \cite{BC} that 
if $u_0\in \Sigma \cap \mathcal{B}_{\omega}^{BC}$, 
then the solution $u(t)$ of \eqref{nls1} with $u(0)=u_0$ blows up in finite time, 
where 
$$\mathcal{B}_{\omega}^{BC}=
\{v\in H^1(\R):S_{\omega}(v)<d(\omega),~ P(v)<0\}.$$
We remark that \eqref{bc1} does not hold for \eqref{nls} with $\gamma>0$. 

On the other hand, 
Zhang \cite{zhang} and Le Coz \cite{lec} gave an alternative proof of the result 
of Berestycki and Cazenave \cite{BC} for \eqref{nls1}. 
Instead of solving the minimization problem \eqref{bc1}, 
they \cite{zhang,lec} proved that  
\begin{align}
d(\omega)\le \inf\{S_{\omega}(v):
v\in H^1(\R), ~ P(v)=0, ~ K_{\omega}(v)<0\}
\label{zl}
\end{align}
holds for all $\omega>0$ if $p\ge 5$. 
Using this fact, it is proved in \cite{zhang,lec} that 
if $u_0\in \Sigma \cap \mathcal{B}_{\omega}^{ZL}$, 
then the solution $u(t)$ of \eqref{nls1} with $u(0)=u_0$ blows up in finite time, where 
$$\mathcal{B}_{\omega}^{ZL}=
\{v\in H^1(\R):S_{\omega}(v)<d(\omega),~ P(v)<0, ~ K_{\omega}(v)<0\}.$$
Note that this method can be applied to \eqref{nls} for the repulsive potential case $\gamma<0$ 
(see \cite{LFF}), but not for the attractive potential case $\gamma>0$. 

In this paper, we use and modify the idea of Zhang \cite{zhang} and Le Coz \cite{lec}   
to prove Theorem \ref{thm1}. 
For $\omega\in (\gamma^2/4,\infty)$ with $E(\phi_{\omega})>0$, 
we introduce a new set 
\begin{align}\label{B} 
\mathcal{B}_{\omega}=\{v\in H^1(\R): 
0<E(v)<E(\phi_{\omega}), ~ 
&\|v\|_{L^2}^2=\|\phi_{\omega}\|_{L^2}^2, \\
&P(v)<0, ~ K_{\omega}(v)<0\}. 
\nonumber 
\end{align}
Then, we have the following.

\begin{thm}\label{thm2}
Let $\gamma>0$, $p>5$, $\omega>\gamma^2/4$, 
and assume that $\phi_{\omega}$ satisfies $E(\phi_{\omega})>0$. 
If $u_0\in \Sigma \cap \mathcal{B}_{\omega}$, 
then the solution $u(t)$ of \eqref{nls} with $u(0)=u_0$ blows up in finite time. 
\end{thm}

The rest of the paper is organized as follows. 
In Section 2, 
using the same method as in our previous paper \cite{OY}, 
we prove Theorem \ref{thm2}. 
In Section 3, we show that 
the sufficient condition $E(\phi_{\omega})>0$ in Theorem \ref{thm2} 
holds if and only if $\omega>\omega_1$, 
and prove Theorem \ref{thm1} using Theorem \ref{thm2}. 
Finally, in Section 4, we compare our sufficient condition for strong instability 
with some known results for orbital instability. 

\section{Proof of Theorem \ref{thm2}}

In this section, we prove Theorem \ref{thm2}. 
As we have already noticed in \S 1, 
the proof of Theorem \ref{thm2} for \eqref{nls} is almost the same as 
that for \eqref{nls0} given in \cite{OY}. 
For the sake of completeness, we repeat the argument in \cite{OY}. 

Throughout this section, we assume that 
$$\gamma>0, \quad p>5, \quad \omega>\frac{\gamma^2}{4}, \quad  E(\phi_{\omega})>0.$$
Recall that $\alpha:=\dfrac{p-1}{2}>2$, and 
for the scaling $v^{\lambda}(x)=\lambda^{1/2} v(\lambda x)$, we have 
\begin{align}
&E(v^{\lambda})
=\frac{\lambda^2}{2}\|\p_x v\|_{L^2}^2
-\frac{\gamma \lambda}{2}|v(0)|^2
-\frac{\lambda^{\alpha}}{p+1}\|v\|_{L^{p+1}}^{p+1}, 
\label{r1}\\
&P(v^{\lambda})
=\lambda^2 \|\p_x v\|_{L^2}^2
-\frac{\gamma \lambda}{2}|v(0)|^2
-\frac{\alpha \lambda^{\alpha}}{p+1}\|v\|_{L^{p+1}}^{p+1} 
=\lambda \p_{\lambda} E(v^{\lambda}), 
\label{r2} \\
&K_{\omega}(v^{\lambda})
=\lambda^2\|\p_x v\|_{L^2}^2+\omega\|v\|_{L^2}^2
-\gamma \lambda|v(0)|^2
-\lambda^{\alpha}\|v\|_{L^{p+1}}^{p+1}. 
\label{r3}
\end{align}

\centerline{\includegraphics[scale=0.8]{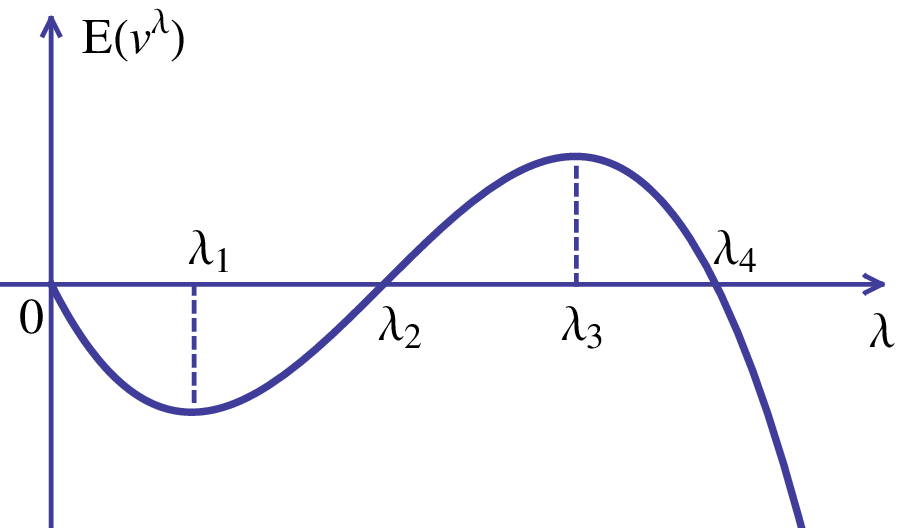}}
\begin{center}
Figure 1. \hspace{1mm} The graph of $\lambda\mapsto E(v^{\lambda})$ 
for the case $E(v)>0$. 
\end{center}

\begin{lmm}\label{lemma1}
If $v\in H^1(\R)$ satisfies $E(v)>0$, 
then there exist $\lambda_k=\lambda_k(v)$ $(k=1,2,3,4)$ such that 
$0<\lambda_1<\lambda_2<\lambda_3<\lambda_4$ and 
\begin{itemize}
\item $E(v^{\lambda})$ is decreasing in $(0,\lambda_1)\cup (\lambda_3,\infty)$, 
and increasing in $(\lambda_1,\lambda_3)$. 
\item $E(v^{\lambda})$ is negative in $(0,\lambda_2)\cup (\lambda_4,\infty)$, 
and positive in $(\lambda_2,\lambda_4)$. 
\item $E(v^{\lambda})<E(v^{\lambda_3})$ 
for all $\lambda\in (0,\lambda_3)\cup (\lambda_3,\infty)$.  
\end{itemize}
\end{lmm}

\begin{proof}
Since $\gamma>0$, $\alpha>2$ and $E(v)>0$, 
the conclusion is easily verified by drawing the graph of \eqref{r1} 
(see Figure 1). 
\end{proof}

\begin{lmm}\label{lemma2}
If $v\in H^1(\R)$ satisfies 
$E(v)>0$, $K_{\omega}(v)<0$ and $P(v)=0$, 
then $d(\omega)<S_{\omega}(v)$. 
\end{lmm}

\begin{proof}
We consider two functions 
$f(\lambda)=K_{\omega}(v^{\lambda})$ and $g(\lambda)=E(v^{\lambda})$. 

Since $f(0)=\omega \|v\|_{L^2}^2>0$ and $f(1)=K_{\omega}(v)<0$, 
there exists $\lambda_0\in (0,1)$ such that 
$K_{\omega}(v^{\lambda_0})=0$. 
Moreover, since $v^{\lambda_0}\ne 0$, it follows from \eqref{gs1} that 
$d(\omega)\le S_{\omega}(v^{\lambda_0})$.

On the other hand, 
since $g'(1)=P(v)=0$ and $g(1)=E(v)>0$, 
it follows from Lemma \ref{lemma1} that 
$\lambda_3=1$ and $g(\lambda)<g(1)$ for all $\lambda\in (0,1)$. 

Thus, we have 
$E(v^{\lambda_0})<E(v)$, and 
\begin{align*}
d(\omega)\le S_{\omega}(v^{\lambda_0}) 
=E(v^{\lambda_0})+\frac{\omega}{2}\|v^{\lambda_0}\|_{L^2}^2
<E(v)+\frac{\omega}{2}\|v\|_{L^2}^2=S_{\omega}(v). 
\end{align*}
This completes the proof. 
\end{proof}

\begin{lmm}\label{invariant}
The set 
\begin{align*}
\mathcal{B}_{\omega}=\{v\in H^1(\R): 
0<E(v)<E(\phi_{\omega}), ~
\|v\|_{L^2}^2=\|\phi_{\omega}\|_{L^2}^2, ~ P(v)<0, ~ K_{\omega}(v)<0\}
\end{align*}
is invariant under the flow of \eqref{nls}. 
That is, if $u_0\in \mathcal{B}_{\omega}$, 
then the solution $u(t)$ of \eqref{nls} with $u(0)=u_0$ satisfies 
$u(t)\in \mathcal{B}_{\omega}$ for all $t\in [0,T_{\max})$. 
\end{lmm}

\begin{proof}
Let $u_0\in \mathcal{B}_{\omega}$ and 
let $u(t)$ be the solution of \eqref{nls} with $u(0)=u_0$. 
Then, by the conservation laws \eqref{conservation}, 
we have 
$$0<E(u(t))=E(u_0)<E(\phi_{\omega}), \quad 
\|u(t)\|_{L^2}^2=\|u_0\|_{L^2}^2=\|\phi_{\omega}\|_{L^2}^2$$
for all $t\in [0,T_{\max})$. 

Next, we prove that $K_{\omega}(u(t))<0$ for all $t\in [0,T_{\max})$. 
Suppose that this were not true. 
Then, since $K_{\omega}(u_0)<0$ and 
$t\mapsto K_{\omega}(u(t))$ is continuous on $[0,T_{\max})$, 
there exists $t_1\in (0,T_{\max})$ such that 
$K_{\omega}(u(t_1))=0$. 
Moreover, since $u(t_1)\ne 0$, by \eqref{gs1}, we have 
$d(\omega)\le S_{\omega}(u(t_1))$. 
Thus, we have 
\begin{align*}
d(\omega)\le S_{\omega}(u(t_1))
=E(u_0)+\frac{\omega}{2}\|u_0\|_{L^2}^2
<E(\phi_{\omega})+\frac{\omega}{2}\|\phi_{\omega}\|_{L^2}^2=d(\omega). 
\end{align*}
This is a contradiction. 
Therefore, $K_{\omega}(u(t))<0$ for all $t\in [0,T_{\max})$. 

Finally, we prove that $P(u(t))<0$ for all $t\in [0,T_{\max})$. 
Suppose that this were not true. 
Then, there exists $t_2\in (0,T_{\max})$ such that $P(u(t_2))=0$. 
Since $E(u(t_2))>0$ and $K_{\omega}(u(t_2))<0$, 
it follows from Lemma \ref{lemma2} that $d(\omega)<S_{\omega}(u(t_2))$. 
Thus, we have 
\begin{align*}
d(\omega)<S_{\omega}(u(t_2))
=E(u_0)+\frac{\omega}{2}\|u_0\|_{L^2}^2
<E(\phi_{\omega})+\frac{\omega}{2}\|\phi_{\omega}\|_{L^2}^2=d(\omega). 
\end{align*}
This is a contradiction. 
Therefore, $P(u(t))<0$ for all $t\in [0,T_{\max})$. 
\end{proof}

\begin{lmm}\label{EP}
For any $v\in \mathcal{B}_{\omega}$, 
$$E(\phi_{\omega})\le E(v)-P(v).$$
\end{lmm}

\begin{proof}
Since $K_{\omega}(v)<0$, as in the proof of Lemma \ref{lemma2}, 
there exists $\lambda_0\in (0,1)$ such that 
$S_{\omega}(\phi_{\omega})=d(\omega)\le S_{\omega}(v^{\lambda_0})$. 
Moreover, since $\|v^{\lambda_0}\|_{L^2}^2
=\|v\|_{L^2}^2=\|\phi_{\omega}\|_{L^2}^2$, 
we have 
\begin{equation}\label{s1}
E(\phi_{\omega})\le E(v^{\lambda_0}).
\end{equation}

On the other hand, 
since $P(v^{\lambda})=\lambda \p_{\lambda} E(v^{\lambda})$, 
$P(v)<0$ and $E(v)>0$, it follows from Lemma \ref{lemma1} that 
$\lambda_3<1<\lambda_4$. 
Moreover, since $\p_{\lambda}^2 E(v^{\lambda})<0$ for $\lambda\in [\lambda_3,\infty)$, 
by a Taylor expansion, we have 
\begin{equation}\label{s2}
E(v^{\lambda_3}) \le E(v)+(\lambda_3-1) P(v) \le E(v)-P(v).
\end{equation}

Finally, by \eqref{s1}, \eqref{s2} and the third property of Lemma \ref{lemma1}, 
we have 
\begin{align*}
E(\phi_{\omega})\le 
E(v^{\lambda_0})\le E(v^{\lambda_3}) \le E(v)-P(v).
\end{align*}
This completes the proof. 
\end{proof}

Now we give the proof of Theorem \ref{thm2}. 

\begin{proof}[Proof of Theorem \ref{thm2}]
Let $u_0\in \Sigma \cap \mathcal{B}_{\omega}$ 
and let $u(t)$ be the solution of \eqref{nls} with $u(0)=u_0$. 
Then, by Lemma \ref{invariant}, 
$u(t)\in \mathcal{B}_{\omega}$ for all $t\in [0,T_{\max})$. 

Moreover, by Lemma \ref{EP} and the virial identity, we have 
\begin{align*}
&P(u(t))\le E(u(t))-E(\phi_{\omega})
 = E(u_0)-E(\phi_{\omega}), \\
&\frac{d^2}{dt^2}\|xu(t)\|_{L^2}^2
=8P(u(t))\le 8\{E(u_0)-E(\phi_{\omega})\}
\end{align*}
for all $t\in [0,T_{\max})$. 
Since $E(u_0)<E(\phi_{\omega})$, this implies $T_{\max}<\infty$. 
This completes the proof. 
\end{proof}

\section{Proof of Theorem \ref{thm1}}

First, we prove the following lemma.

\begin{lmm}\label{lemma5}
Let $\gamma>0$, $p>5$ and $\omega>\gamma^2/4$. 
Let $\omega_1(p,\gamma)$ be the number defined in Theorem \ref{thm1}. 
Then, $E(\phi_{\omega})>0$ if and only if $\omega>\omega_1(p,\gamma)$. 
\end{lmm}

\begin{proof}
Since $P(\phi_{\omega})=0$, 
we see that $E(\phi_{\omega})>0$ if and only if 
\begin{equation}\label{Epq}
\gamma |\phi_{\omega}(0)|^2
<\frac{p-5}{p+1}\|\phi_{\omega}\|_{L^{p+1}}^{p+1}. 
\end{equation}
Moreover, by \eqref{for} and \eqref{xi-omega}, we have 
\begin{align}
|\phi_{\omega}(0)|^2
&=\bigg\{
\frac{(p+1)\omega}{2}{\rm sech}^2 
\left(\tanh^{-1} \xi(\omega,\gamma) \right) 
\bigg\}^{\frac{2}{p-1}} 
\nonumber \\
&=\left(\frac{(p+1)\omega}{2}\, \left\{1-\xi(\omega,\gamma)^2\right\}\right)^{\frac{2}{p-1}}, 
\label{RR1}
\end{align}
\begin{align}
\|\phi_{\omega} \|_{L^{p+1}}^{p+1} 
&=2\int_{0}^{\infty} 
\bigg\{\frac{(p+1)\omega}{2} {\rm sech}^2
\left(\frac{(p-1)\sqrt{\omega}}{2}\,x+\tanh^{-1} \xi  (\omega,\gamma) \right) 
\bigg\}^{\frac{p+1}{p-1}}dx 
\nonumber \\
&=\frac{4}{(p-1)\sqrt{\omega}} 
\left(\frac{(p+1)\omega}{2}\right)^{\frac{p+1}{p-1}}
\int_{\tanh^{-1} \xi  (\omega,\gamma)}^{\infty} ({\rm sech}^2 y)^{\frac{p+1}{p-1}}\,dy 
\nonumber \\
&=\frac{4}{(p-1)\sqrt{\omega}} \left( \frac{(p+1)\omega}{2} \right)^{\frac{p+1}{p-1}}
\int_{\xi(\omega,\gamma)}^{1} (1-s^2)^{\frac{2}{p-1}} \, ds, 
\label{RR2}
\end{align}
where we used \eqref{CI}. 
Thus, we see that \eqref{Epq} is equivalent to 
\begin{equation}\label{Ep5}
\frac{p-5}{p-1} \int_{\xi(\omega,\gamma)}^{1}(1-s^2)^{\frac{2}{p-1}} \, ds
>\xi(\omega,\gamma) \{1-\xi(\omega,\gamma)^2\}^{\frac{2}{p-1}}.
\end{equation}
Moreover, by elementary compuations, we see that 
the equation \eqref{xi1} has a unique solution $\xi_1(p)$ for each $p>5$, 
and that \eqref{Ep5} is equivalent to $\xi(\omega,\gamma)<\xi_1(p)$. 
This completes the proof. 
\end{proof}

Now we give the proof of Theorem \ref{thm1}. 

\begin{proof}[Proof of Theorem \ref{thm1}]
Let $\omega\in (\omega_1,\infty)$. 
Then, by Lemma \ref{lemma5}, $E(\phi_{\omega})>0$. 

For $\lambda>0$, we consider the scaling 
$\phi_{\omega}^{\lambda}(x)=\lambda^{1/2}\phi_{\omega}(\lambda x)$, 
and prove that there exists $\lambda_0\in (1,\infty)$ such that 
$\phi_{\omega}^{\lambda}\in \mathcal{B}_{\omega}$ 
for all $\lambda\in (1,\lambda_0)$. 

First, we have $\|\phi_{\omega}^{\lambda}\|_{L^2}^2=\|\phi_{\omega}\|_{L^2}^2$ 
for all $\lambda>0$. 
Next, since $P(\phi_{\omega})=0$ and $E(\phi_{\omega})>0$, 
by Lemma \ref{lemma1} and \eqref{r2},  
there exists $\lambda_4>1$ such that 
$$0<E(\phi_{\omega}^{\lambda})<E(\phi_{\omega}), \quad 
P(\phi_{\omega}^{\lambda})<0$$ 
for all $\lambda\in (1,\lambda_4)$. 
Finally, since $P(\phi_{\omega})=0$, we have 
\begin{align*}
\p_{\lambda} K_{\omega}(\phi_{\omega}^{\lambda})\big|_{\lambda=1}
=-\frac{(p-1)\alpha}{p+1}\|\phi_{\omega}\|_{L^{p+1}}^{p+1}<0.
\end{align*}
Since $K_{\omega}(\phi_{\omega})=0$, 
there exists $\lambda_0\in (1,\lambda_4)$ such that 
$K_{\omega}(\phi_{\omega}^{\lambda})<0$ for all $\lambda\in (1,\lambda_0)$. 

Therefore, $\phi_{\omega}^{\lambda}\in \mathcal{B}_{\omega}$ 
for all $\lambda\in (1,\lambda_0)$. 
Moreover, since $\phi_{\omega}^{\lambda}\in \Sigma$ for $\lambda>0$,  
it follows from Theorem \ref{thm2} that for any $\lambda\in (1,\lambda_0)$, 
the solution $u(t)$ of \eqref{nls} with $u(0)=\phi_{\omega}^{\lambda}$ blows up 
in finite time. 

Finally, since $\displaystyle{\lim_{\lambda\to 1}
\|\phi_{\omega}^{\lambda}-\phi_{\omega}\|_{H^1}=0}$, 
the proof is completed. 
\end{proof}

\section{Final Remarks}

In \cite{GR, oht2}, a sufficient condition on orbital instability of standing waves 
for some nonlinear Schr\"odinger equations is given. 
The sufficient condition by \cite{GR, oht2} 
is different from that given by Shatah and Strauss \cite{SS}, 
and it is applicable to \eqref{nls}. 

More precisely, by \cite{GR,oht2}, we see that if 
\begin{equation}\label{sc2}
\partial_{\lambda}^2E(\phi_{\omega}^{\lambda}) \big|_{\lambda=1}<0,
\end{equation}
then the standing wave solution $e^{i \omega t}\phi_{\omega}$ of \eqref{nls} is orbitally unstable, 
where $\phi_{\omega}^{\lambda}(x)=\lambda^{1/2}\phi_{\omega}(\lambda x)$. 

In this section, we compare the sufficient condition 
on strong instability $E(\phi_{\omega})>0$ in Theorem \ref{thm2} 
with sufficient conditions on orbital instability  $\partial_{\omega} \|\phi_{\omega}\|_{L^2}^2<0$ 
by \cite{SS} and \eqref{sc2} by \cite{GR, oht2}. 
 
By \eqref{r1}, we have 
\begin{equation}
\partial_{\lambda}^2E(\phi_{\omega}^{\lambda}) \big|_{\lambda=1}
=\|\partial_x \phi_{\omega}\|_{L^2}^{2}
-\frac{(p-1)(p-3)}{4(p+1)} \|\phi_{\omega}\|_{L^{p+1}}^{p+1}.
\end{equation}
Since $P(\phi_{\omega})=0$, we see that \eqref{sc2} is equivalent to 
\begin{align}\label{Epq2}
\gamma |\phi_{\omega}(0)|^2
<\frac{(p-1)(p-5)}{2(p+1)}\|\phi_{\omega}\|_{L^{p+1}}^{p+1}.
\end{align}
Moreover, by \eqref{RR1} and \eqref{RR2}, 
we see that \eqref{Epq2} is equivalent to 
\begin{equation}\label{Ep7}
\frac{p-5}{2} \int_{\xi(\omega,\gamma)}^{1}(1-s^2)^{\frac{2}{p-1}} \, ds
>\xi(\omega,\gamma) \{1-\xi(\omega,\gamma)^2\}^{\frac{2}{p-1}}.
\end{equation}
For $p>5$, let $\xi_2(p)\in (0,1)$ be a unique solution of 
\begin{equation}\label{xi2}
\frac{p-5}{2} \int_{\xi}^{1}(1-s^2)^{\frac{2}{p-1}} \, ds
=\xi\, (1-\xi^2)^{\frac{2}{p-1}} \quad (0<\xi<1).
\end{equation}
Then, we see that \eqref{Ep7} is equivalent to 
$\xi (\omega,\gamma)<\xi_2(p)$, 
and that \eqref{sc2} holds if and only if $\omega>\omega_2(p,\gamma):=\gamma^2/[4\xi_2(p)^2]$. 

On the other hand, as we mentioned in Remark \ref{rem1}, 
the condition $\partial_{\omega} \|\phi_{\omega}\|_{L^2}^2<0$ 
holds if and only if $\omega>\omega_0(p,\gamma)$, 
while as we proved in Lemma \ref{lemma5}, 
the condition $E(\phi_{\omega})>0$ holds if and only if $\omega>\omega_1(p,\gamma)$. 
Recall that for $j=0,1,2$, 
$$\omega_j(p,\gamma)=\frac{\gamma^2}{4\xi_j(p)^2},$$
and $\xi_0(p)$, $\xi_1(p)$ and $\xi_2(p)$ 
are the unique solutions of \eqref{xi0}, \eqref{xi1} and \eqref{xi2},  respectively. 

For $p>5$, we see that $\xi_1(p)<\xi_2(p)<\xi_0(p)$, 
and that $\omega_0(p,\gamma)<\omega_2(p,\gamma)<\omega_1(p,\gamma)$. 
The graphs of the functions $\xi_0(p)$, $\xi_1(p)$ and $\xi_2(p)$ 
are given in Figure 2 for $5<p\le 10$ and in Figure 3 for $10\le p\le 30$. 

For $\omega\in [\omega_0(p,\gamma), \omega_1(p,\gamma)]$, 
the standing wave solution $e^{i\omega t} \phi_{\omega}$ of \eqref{nls} is orbitally unstable, 
but we do not known whether it is strongly unstable or not. 
However, it seems natural to conjecture that 
$e^{i\omega t} \phi_{\omega}$ is strongly unstable at least for $\omega>\omega_2(p,\gamma)$. 
Note that some numerical results are given 
in Section 6.1.3 of \cite{LFF} 
for the case where $p=6$, $\gamma=1$ and $\omega=4$. 
We remark that 
$$\xi_1 (6)=0.137\cdots<\xi(4,1)=0.25<\xi_2 (6)=0.279\cdots,$$ 
and that $\omega_2(6,1)<4<\omega_1(6,1)$ for this case. 

\vspace{5mm}
\centerline{\includegraphics[scale=0.6]{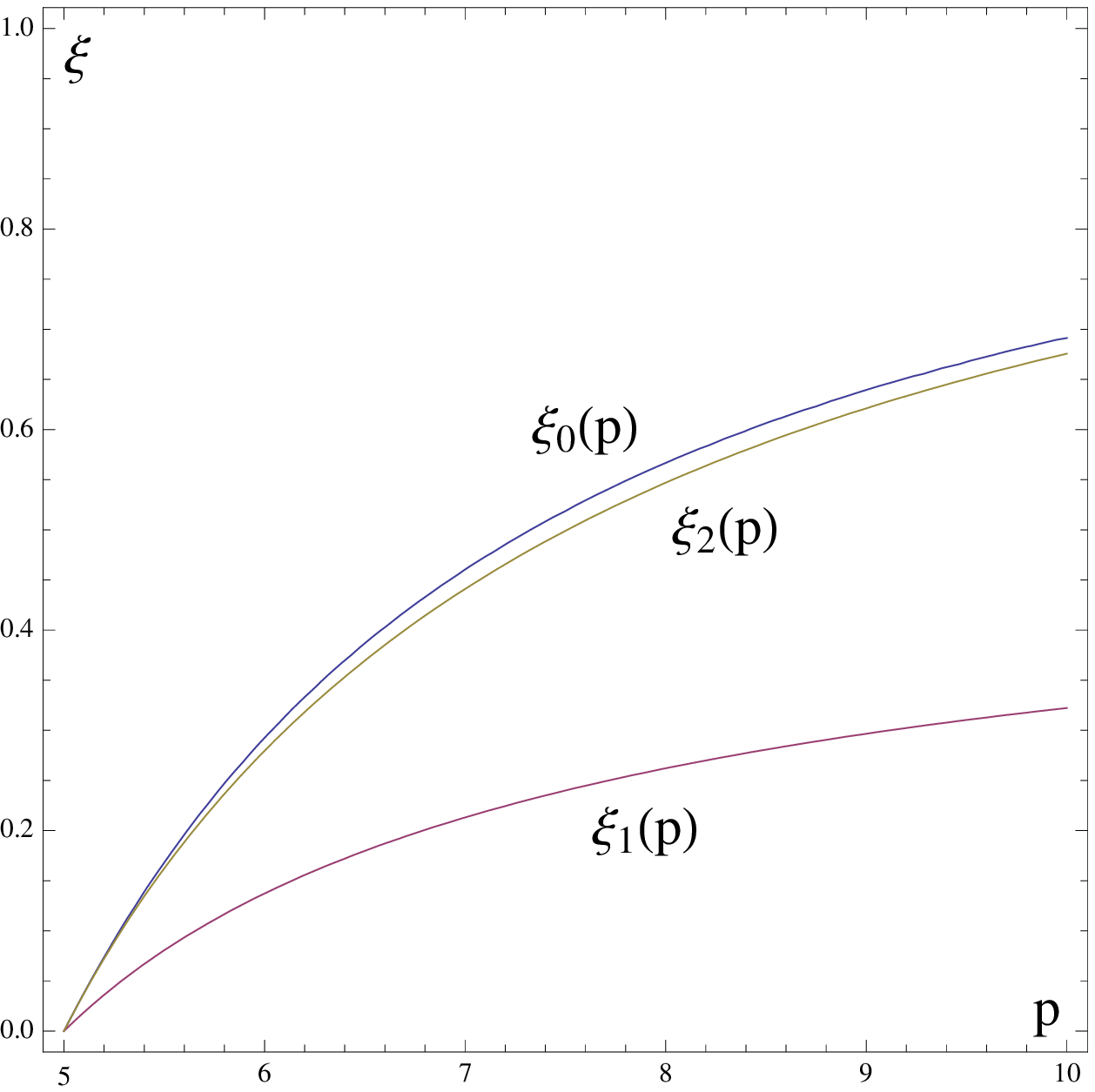}}
\begin{center}
Figure 2. \hspace{1mm} The graphs of $\xi_0(p)$, $\xi_1(p)$ and $\xi_2(p)$ for $5<p\le 10$. 
\end{center}

\centerline{\includegraphics[scale=0.6]{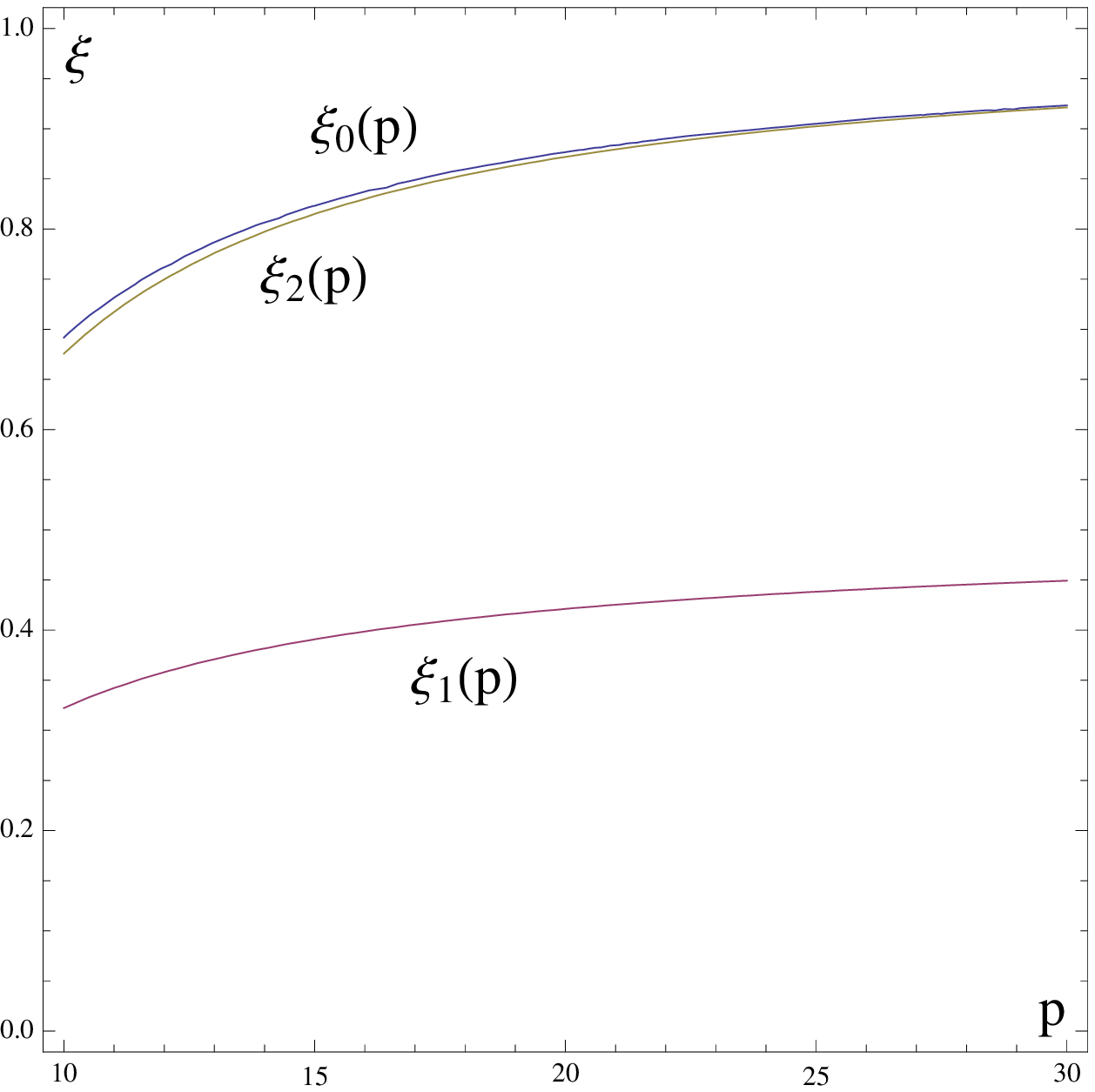}}
\begin{center}
Figure 3. \hspace{1mm} The graphs of $\xi_0(p)$, $\xi_1(p)$ and $\xi_2(p)$ for $10\le p\le 30$. 
\end{center}


\begin{thebibliography}{99}
%
%
\bibitem{BC}H. Berestycki and T. Cazenave, 
Instabilit\'{e} des \'{e}tats stationnaires dans les \'{e}quations 
de Schr\"{o}dinger et de Klein-Gordon non lin\'{e}aires, 
\textit{C. R. Acad. Sci. Paris S\'{e}r. I Math.}, 
\textbf{293} (1981), 489--492. 

\bibitem{caz}T. Cazenave,
Semilinear Schr\"odinger equations, 
\textit{Courant Lecture Notes in Mathematics} 10, 
Amer. Math. Soc., 2003. 

\bibitem{CL}T. Cazenave and P.-L. Lions,
Orbital stability of standing waves for 
some nonlinear Schr\"odinger equations, 
\textit{Comm. Math. Phys.}, 
\textbf{85} (1982), 549--561.

\bibitem{FJ}R. Fukuizumi, R. Jeanjean,
Stability of standing waves for a nonlinear Schr\"odinger equation 
with a repulsive Dirac delta potential, 
\textit{Discrete Contin. Dyn. Syst.}, 
\textbf{21} (2008), 121--136. 

\bibitem{FOO}R. Fukuizumi, M. Ohta and T. Ozawa, 
Nonlinear Schr\"odinger equation with a point defect, 
\textit{Ann. Inst. H. Poincar\'e, Anal. Non Lin\'eaire}, 
\textbf{25} (2008), 837--845.

\bibitem{GR}J.~M.~Gon\c{c}alves Ribeiro, 
Instability symmetric stationary states for some nonlinear Schr\"odinger equations 
with an external magnetic field, 
\textit{Ann. Inst. H. Poincar\'e, Phys. Th\'eor.}, 
\textbf{54} (1991), 403--433.

\bibitem{GHW}R. H. Goodman, P. J. Holmes, and M. I. Weinstein, 
Strong NLS soliton-defect interactions, 
\textit{Phys. D}, 
\textbf{192} (2004), 215--248.

\bibitem{GSS1}M. Grillakis, J. Shatah and W. Strauss, 
Stability theory of solitary waves in the presence of symmetry I, 
\textit{J. Funct. Anal.}, 
\textbf{74} (1987) 160--197. 

\bibitem{GSS2}M. Grillakis, J. Shatah and W. Strauss, 
Stability theory of solitary waves in the presence of symmetry II,  
\textit{J. Funct. Anal.}, 
\textbf{94} (1990), 308--348. 

\bibitem{HMZ1} J. Holmer, J. Marzuola and M. Zworski, 
Fast soliton scattering by delta impurities,
\textit{Comm. Math. Phys.}, 
\textbf{274} (2007), 187--216.

\bibitem{HMZ2} J. Holmer, J. Marzuola and M. Zworski, 
Soliton splitting by external delta potentials, 
\textit{J. Nonlinear Sci.}, 
\textbf{17} (2007), 349--367.

\bibitem{HZ} J. Holmer and M. Zworski, 
Slow soliton interaction with delta impurities,
\textit{J. Mod. Dyn.},  
\textbf{1} (2007), 689--718.

\bibitem{KO}M. Kaminaga and M. Ohta, 
Stability of standing waves for nonlinear Schr\"odinger equation 
with attractive delta potential and repulsive nonlinearity, 
\textit{Saitama Math. J.}, 
\textbf{26} (2009), 39--48. 

\bibitem{lec}S. Le Coz, 
A note on Berestycki-Cazenave's classical instability result 
for nonlinear Schr\"odinger equations, 
\textit{Adv. Nonlinear Stud.}, 
\textbf{8} (2008), 455--463. 

\bibitem{LFF}S. Le Coz, R. Fukuizumi, G. Fibich, B. Ksherim and Y. Sivan, 
Instability of bound states of a nonlinear Schr\"odinger equation with a Dirac potential, 
\textit{Phys. D}, 
\textbf{237} (2008), 1103--1128.

\bibitem{oht2}M. Ohta,
Instability of standing waves for the generalized Davey-Stewartson system,
\textit{Ann. Inst. H. Poincar\'e, Phys. Th\'eor.}, 
\textbf{62} (1995), 69--80.

\bibitem{oht4}M. Ohta, 
Instability of bound states for abstract nonlinear Schr\"odinger equations, 
\textit{J. Funct. Anal.}, 
\textbf{261} (2011), 90--110.

\bibitem{OY}M. Ohta and T. Yamaguchi, 
Strong instability of standing waves for nonlinear Schr\"odinger equations 
with double power nonlinearity, 
\textit{preprint}, arXiv:1407.0905. 

\bibitem{sha}J. Shatah, 
Stable standing waves of nonlinear Klein-Gordon equations, 
\textit{Comm. Math. Phys.}, 
\textbf{91} (1983), 313--327. 

\bibitem{SS}J. Shatah and W. Strauss,
Instability of nonlinear bound states, 
\textit{Comm. Math. Phys.}, 
\textbf{100} (1985), 173--190.

\bibitem{wei} M.I. Weinstein, 
Lyapunov stability of ground states of nonlinear dispersive evolution equations, 
\textit{Comm. Pure Appl. Math.}, 
\textbf{39} (1986), 51--67.

\bibitem{zhang}J. Zhang, 
Cross-constrained variational problem and nonlinear Schr\"{o}dinger equation, 
\textit{Foundations of computational mathematics (Hong Kong, 2000)}, 
World Sci. Publ., River Edge, NJ, 2002, pp. 457--469. 

\end{thebibliography}
\end{document}